\documentclass[11pt]{amsart}

\usepackage{amsmath, amssymb, amscd, cancel, graphicx, paralist, soul, stmaryrd}

\usepackage{mathdots}

\usepackage[all]{xy}
\usepackage{multirow}
\usepackage{longtable}
\usepackage{array}

\newtheorem{thm}{Theorem}[section]

\newtheorem{cor}[thm]{Corollary}
\newtheorem{lem}[thm]{Lemma}

\newtheorem{prop}[thm]{Proposition}

\def\bQ{{\mathbb Q}}

\def\bZ{{\mathbb Z}}

\def\del{{\partial}}

\def\int{{\text{int}}}
\def\lk{{\mathrm{lk}}}
\def\mod{{\textup{mod} \;}}

\begin{document}

\title[Alternating links and definite surfaces]%
{Alternating links and definite surfaces}

\author[Joshua Evan Greene]{Joshua Evan Greene}

\address{Department of Mathematics, Boston College\\ Chestnut Hill, MA 02467}

\email{joshua.greene@bc.edu}

\maketitle

\medskip

\noindent {\bf Abstract.} We establish a characterization of alternating links in terms of definite spanning surfaces.  We apply it to obtain a new proof of Tait's conjecture that reduced alternating diagrams of the same link have the same crossing number and writhe. We also deduce a result of Banks and Hirasawa-Sakuma about Seifert surfaces for special alternating links.  The appendix, written by Juh\'asz and Lackenby, applies the characterization to derive an exponential time algorithm for alternating knot recognition.

\medskip

\noindent {\bf MSC classes:}  05C21, 05C50, 11H55, 57M15, 57M25, 57M27


\section{Introduction.}

\begin{center}
{\em ``What is an alternating knot?"}  -- Ralph Fox
\end{center}

A link diagram is alternating if its crossings alternate over and under around each link component, and a link is alternating if it admits an alternating diagram.  The opening question due to Fox seeks a characterization of alternating links in terms intrinsic to the link complement \cite[p.32]{lickorish:book}.  We establish such a characterization here in terms of definite spanning surfaces.

To describe it, a compact surface in a $\bZ / 2 \bZ$ homology sphere carries a natural pairing on its ordinary first homology group, mildly generalizing a definition by Gordon and Litherland \cite{gl:sig}.  An alternating diagram of a non-split alternating link in $S^3$ yields an associated pair of black and white chessboard spanning surfaces for the link, and their pairings are respectively negative and positive definite.  We establish the following converse:

\begin{thm}
\label{t: alt definite}
Let $L$ be a link in a $\bZ/2\bZ$ homology sphere with irreducible complement, and suppose that it bounds both a negative definite surface and a positive definite surface.  Then $L$ is a non-split alternating link in $S^3$, and it has an alternating diagram whose associated chessboard surfaces are isotopic rel boundary to the two given surfaces.
\end{thm}

The characterization given in Theorem \ref{t: alt definite} is compelling in that it leads to new, conceptual proofs of one of Tait's conjectures (c.~1876), amongst other applications.
\begin{thm}
\label{t: tait}
Any two connected, reduced, alternating diagrams of the same link have the same crossing number and writhe.
\end{thm}
\noindent
A diagram is reduced if every crossing touches four distinct regions.
Theorem \ref{t: tait} was originally proven independently by Kauffman, Murasugi, and Thistlethwaite using properties of the Jones polynomial, shortly following its discovery \cite{kauff:altknots, murasugi:altknots, thistle:spanning}.  By contrast, the proof we give is based on more classical topological constructions and some basic facts about flows on planar graphs.

A connected, oriented alternating diagram is special if one of the associated spanning surfaces is orientable.  Seifert's algorithm outputs this surface when applied to such a diagram.  An oriented alternating link is special if it has a special alternating diagram.  Theorem \ref{t: alt definite} has the following straightforward consequence, first established by Banks and Hirasawa-Sakuma using geometric methods \cite{banks:special,hirasawasakuma}.

\begin{cor}
\label{c: seifert}
A Seifert surface for a special alternating link $L$ has minimum genus if and only if it is obtained by applying Seifert's algorithm to a special alternating diagram of $L$.
\end{cor}

At the time of a colloquium about these results at the University of Texas, Austin in January 2015, Tye Lidman observed the following immediate Corollary to Theorem \ref{t: alt definite}.

\begin{cor}
\label{c: tye}
An amphichiral knot with a definite spanning surface is alternating. \qed
\end{cor}

\noindent
Following a conference talk about these results at Princeton University in June 2015, Andr\'as Juh\'asz and Marc Lackenby applied Theorem \ref{t: alt definite} to the algorithmic detection of prime alternating knots.  With their gracious permission, we include their result and proof.

\begin{thm}
\label{t: algorithm}
Given a diagram of a prime knot $K$ with $c$ crossings, there exists an $\exp(c^2)$ time algorithm to decide whether $K$ is alternating.
\end{thm}


\noindent
{\bf Organization.}
Section \ref{ss: gl} reviews the work of Gordon and Litherland on their eponymous pairing and its applications to link signatures, and then points out how their definition and results generalize to the case of a $\bZ / 2 \bZ$ homology sphere.  Section \ref{ss: definite} defines definite surfaces and collects their basic properties.  Section \ref{ss: characterization} applies this preparatory material in order to prove Theorem \ref{t: alt definite} and deduce Corollary \ref{c: seifert}.  Section \ref{s: tait} develops the elementary theory of flows on planar graphs in order to deduce Theorem \ref{t: tait} from Theorem \ref{t: alt definite}.  Finally, the Appendix, written by Juh\'asz and Lackenby, contains the proof of Theorem \ref{t: algorithm}.

\noindent
{\bf Convention.}
We use integer coefficients for all chain groups and homology groups of graphs and surfaces.


\subsection*{Acknowledgments.} My foremost thanks go to Andr\'as Juh\'asz,  Marc Lackenby, and Tye Lidman for their valuable contributions to this paper.  
Thanks to the pair of Paolo Lisca and Brendan Owens, and also to Yi Ni, who independently suggested that Theorem \ref{t: alt definite} should hold for a broader class of 3-manifolds than integer homology spheres, for which it was initially proven.  Thanks lastly to John Baldwin, Peter Feller, John Luecke, and Morwen Thistlethwaite for many enjoyable and stimulating discussions.  This work was supported by NSF CAREER Award DMS-1455132 and an Alfred P. Sloan Research Fellowship.


\section{The Gordon-Litherland pairing.}
\label{ss: gl}

Generalizing earlier work by several researchers \cite{goeritz,kauffmantaylor,seifert,trotter}, Gordon and Litherland defined a symmetric bilinear pairing on the ordinary first homology group of a compact embedded surface in $S^3$ \cite{gl:sig}.  We recall their definition and their main results, and then we promote their work to the setting of a $\bZ / 2 \bZ$ homology sphere.

Let $Y = S^3$ and $S \subset Y$ a compact, connected, embedded surface.  The unit normal bundle to $S$ embeds as a subspace $N(S) \subset Y \setminus S$ and carries a 2-to-1 covering map
\[
p_S: N(S) \to S.
\]
Given pair of homology classes $a,b \in H_1(S)$, represent them by embedded, oriented multi-curves $\alpha, \beta \subset S$. Define
\[
\langle a,b \rangle_S = \lk(\alpha, p_S^{-1}(\beta)),
\]  
where $\lk$ denotes the linking number.
Gordon and Litherland prove that the pairing $\langle \, , \rangle_S$ establishes a well-defined, symmetric, bilinear pairing
\[
\langle \, , \rangle_S : H_1(S) \times H_1(S) \to \bZ
\]
\cite[Theorem 3 and Proposition 9]{gl:sig}.  When $S$ is orientable, the pairing coincides with the symmetrized Seifert pairing.

They also show how to use the pairing $\langle \, , \rangle_S$ to determine the signature of a link.  Suppose that $S$ is a spanning surface for a link $L$, meaning that $L = \del S$.  The components $K_1, \dots, K_m$ of $L$ define projective homology classes $[K_1],\dots,[K_m] \in H_1(S)/\pm$.  For a projective class $x \in H_1(S)/\pm$, let $|x|_S = \langle x,x \rangle_S$ denote its well-defined self-pairing.  The value $\frac 12 |[K_i]|_S$ equals the framing that $S$ induces on $K_i$.  Let $e(S)$ denote the {\em euler number} $-\frac 12 \sum_{i=1}^m |[K_i]|_S$.  If $L$ is oriented, then let $e(S,L) = -\frac 12 |[L]|_S$.  The two quantities are related by the identity $e(S,L) = e(S) - \mathrm{lk}(L)$, where $\mathrm{lk}(L)$ denotes the total linking number $\sum_{i < j} \lk(K_i,K_j)$.  Lastly, let $\sigma(S)$ denote the signature of the pairing $\langle \, , \rangle_S$.

Gordon and Litherland's result reads as follows in the case that $Y = S^3$ \cite[Corollaries $5'$ and $5''$]{gl:sig}.  As we discuss below, it pertains more generally to the case of $\bZ / 2 \bZ$ homology sphere $Y$.

\begin{thm}
\label{t: gl}
If $S$ is a compact spanning surface for an unoriented link $L \subset Y$, then the quantity
\[
\sigma(S) + \frac 12 e(S)
\]
depends only on $L$, and it coincides with the Murasugi invariant $\xi(L)$ when $Y = S^3$.  If $L$ is oriented, then 
\[
\sigma(S) + \frac 12 e(S,L)
\]
depends only on $L$, and it coincides with the link signature $\sigma(L)$ when $Y = S^3$.
\qed
\end{thm}

\noindent
The Murasugi invariant $\xi(L)$ is the average of the signatures of the different oriented links whose underlying unoriented link is $L$.
Note that if $S$ is a Seifert surface for an oriented link $L$, then $[L]=0 \in H_1(S)/\pm$ and $\langle \, , \rangle_S$ coincides with the symmetrized Seifert pairing.  We therefore recover the familiar definition of the link signature in this case.

Now we turn to the case in which $Y$ is an arbitrary $\bZ / 2 \bZ$ homology sphere.  The preceding summary carries over to this setting, and we highlight the necessary alterations.

The key distinction is that a pair of disjoint, oriented curves $K_1, K_2 \subset Y$ have a {\em rational} linking number $\lk(K_1,K_2)$.  To describe it, orient $Y$ and take a rational Seifert surface $S_1$ that runs $q > 0$ times around $K_1$ and meets $K_2$ transversely.  Then set $\lk(K_1,K_2) = (S_1 \cdot K_2)/q$.  A standard argument shows that this value is independent of the choice of rational Seifert surface, it is symmetric in $K_1$ and $K_2$, and it extends by linearity to a $\bQ$-valued function on pairs of disjoint, oriented links in $Y$.

The Gordon-Litherland pairing in this setting is a pairing
\[
H_1(S) \times H_1(S) \to \bQ
\]
defined exactly as above with respect to the rational linking number.  The proof that it is well-defined and bilinear is straightforward, and the proof of \cite[Proposition 9]{gl:sig} applies directly to show that it is symmetric.  The remaining definitions that go into the statement of Theorem \ref{t: gl} also apply directly to this setting without change.  The proof of Theorem \ref{t: gl} given in \cite[Section 6]{gl:sig} applies as well, with two important notes.  First, the notion of $S^*$-equivalence of spanning surfaces carries over without change, as does the proof of \cite[Proposition 10]{gl:sig}, using the rational linking number.  Second, Proof I of \cite[Theorem 11]{gl:sig} only uses the fact that $S^3$ is a $\bZ / 2 \bZ$ homology sphere.  The salient point is that, in the notation of that proof, $V_0 \cup V_1$ is a $(\mod 2)$ 2-cycle, so there exist $(\mod 2)$ 3-chains $Y_0, Y_1 \subset S^3$ such that $S^3 = Y_0 \cup Y_1$ and $Y_0 \cap Y_1 = \del Y_0 = \del Y_1 = V_0 \cup V_1$.  This decomposition is used implicitly in the assertions made there about the subspaces $M$ and $M'$.  As the same holds for any $\bZ / 2 \bZ$ homology sphere $Y$, the proof adapts simply by substituting $Y$ for $S^3$.

We may take the invariant values $\xi(L)$ and $\sigma(L)$ appearing in Theorem \ref{t: gl} as the natural generalizations of the Murasugi invariant and the oriented link signature of a null-homologous link $L$ in a $\bZ / 2 \bZ$ homology sphere.  The link $L$ with a choice of orientation is null-homologous in this setting, since we assume it bounds a spanning surface $S$, which implies that $[L]=0 \in H_1(Y;\bZ/2\bZ)$, and so $[L] = 0 \in H_1(Y;\bZ)$.  We mention in closing that signatures of oriented links in rational homology spheres were studied in greater generality by Cha and Ko in \cite{chako2002}.


\section{Definite surfaces.}
\label{ss: definite}

A compact, connected surface $S$ in a $\bZ / 2 \bZ$ homology sphere is {\em definite} (either positive or negative) if its Gordon-Litherland pairing is.

\begin{prop}
\label{p: min complexity}
If $S$ is a definite surface with boundary $L$, then $b_1(S)$ is minimal over all spanning surfaces for $L$ with the same euler number as $S$.  Moreover, if $S'$ is such a surface with $b_1(S) = b_1(S')$, then $S'$ is definite and of the same sign as $S$.
\end{prop}

\begin{proof}
The Gordon-Litherland formula implies that all such surfaces have the same signature, whose absolute value therefore bounds from below the first betti number of any such surface.  By definition, this bound is attained by a definite surface.
\end{proof}

\begin{cor}
\label{c: incompressible}
A definite surface is incompressible. \qed
\end{cor}

\noindent
Following Proposition \ref{p: diagram} below, Corollary \ref{c: incompressible} generalizes \cite[Prop 2.3]{mt:tait}, which treats the case of a chessboard surface associated with a reduced alternating diagram of a link in $S^3$.

\begin{lem}
\label{l: subsurface}
If $S$ is definite and $S' \subset S$ is a compact subsurface with connected boundary, then $S'$ is definite.
\end{lem}

\begin{proof}
Since $S$ is definite, any compact subsurface $S''$ is semidefinite: the self-pairings of its homology classes take only one sign, and the self-pairing vanishes precisely on the kernel of the inclusion-induced map $H_1(S'') \to H_1(S)$.  Since $S$ and $\del S'$ are connected, $S/S'$ is a (possibly empty) connected surface with boundary, so $0 = H_2(S/S') \approx H_2(S,S')$.  The long exact sequence of the pair $(S,S')$ now shows that the inclusion-induced map $H_1(S') \to H_1(S)$ injects.  It follows that $S'$ is definite.
\end{proof}

\begin{lem}
\label{l: intersect}
Suppose that $Y$ is a $\bZ / 2 \bZ$ homology sphere, $L \subset Y$ is a link, $X = Y \setminus \overset{\circ}{\nu}(L)$ is irreducible, and $S_\pm \subset Y$ are $\pm$-definite spanning surfaces for $L$.  If $S_+ \cap X$ and $S_- \cap X$ are in minimal position, then $S_+ \cap S_- \cap X$ does not contain a simple closed curve of intersection.
\end{lem}

\begin{proof}
Suppose that $S_+ \cap S_- \cap X$ contains a simple closed curve $\gamma$.
Observe that $S_+ \cap \del \nu(\gamma)$ and $S_- \cap \del \nu(\gamma)$ are parallel on $\del \nu(\gamma)$, and moreover that $S_\pm \cap \del \nu(\gamma)$ is isotopic to $p_{S_{\mp}}^{-1}(\gamma)$ in $Y \setminus \gamma$.  It follows that $0 \le |\gamma|_{S_+} = |\gamma|_{S_-} \le 0$.  Therefore, $\gamma$ is null-homologous in both $S_+$ and $S_-$.  Let $S'_\pm \subset S_\pm$ denote the orientable subsurfaces with $\del S'_\pm = \gamma$.  These surfaces are respectively positive and negative definite by Lemma \ref{l: subsurface}, and $\sigma(S'_+) =\sigma(S'_-) = \sigma(\gamma)$ because they are Seifert surfaces for the knot $\gamma \subset Y$.  Therefore, $0 \le b_1(S'_+) = \sigma(S'_+) = \sigma(S'_-) = - b_1(S'_-) \le 0$, so $S'_+$ and $S'_-$ are disks.
By passing to an innermost disk, we may assume that $S'_+$ and $S'_-$ have disjoint interiors, so their union is a sphere.  
Since $X$ is irreducible, the sphere $S'_+ \cup S'_-$ bounds a ball in $X$.  This ball guides an isotopy that reduces the number of components of $S_+ \cap S_- \cap X$, so $S_+ \cap X$ and $S_- \cap X$ were not in minimal position.  The conclusion of the Lemma now follows.
\end{proof}


\section{Proof of the characterization.}
\label{ss: characterization}

The following Proposition characterizes alternating {\em diagrams} in terms of the definiteness of their associated chessboard surfaces.  It plays a role in the proof of Theorem \ref{t: alt definite} to follow.

\begin{prop}
\label{p: diagram}
Let $D$ denote a connected diagram of a link $L$, and let $B$ and $W$ denote its associated chessboard surfaces.  Then $D$ is alternating if and only if $B$ and $W$ are definite surfaces of opposite signs.
\end{prop}

\begin{proof}

\begin{figure}
\includegraphics[width=5in]{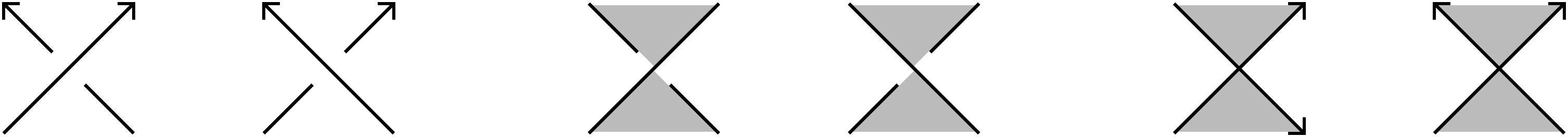}
\caption{A positive and a negative crossing in an oriented link diagram; a type $a$ and a type $b$ crossing in a colored link diagram; and a type I and a type II crossing in an oriented, colored link projection.}
\label{f: crossings}
\end{figure}

Orient $D$.
Referring to Figure \ref{f: crossings}, the value $\frac 12 e(B,L)$ equals the number of crossings that are both of type $b$ and type II minus the number of crossings that are both of type $a$ and type II \cite[Lemma 7]{gl:sig}.  Similarly, $\frac 12 e(W,L)$ equals that number of crossings that are both of type $a$ and type I minus the number of crossings that are both of type $b$ and type I.
Let $a(D)$ and $b(D)$ denote the number of type $a$ and type $b$ crossings, respectively.  It follows that
\[
b(D) - a(D) = \frac 12 e(B,L) - \frac 12 e(W,L).
\]
On the other hand, Theorem \ref{t: gl} gives
\[
\frac 12 e(B,L) - \frac 12 e(W,L) = \sigma(W) - \sigma(B).
\]
Taking the absolute value leads to
\[
|b(D)-a(D)|=|\sigma(W) - \sigma(B)| \le |\sigma(W)| + |\sigma(B)| \le b_1(W) + b_1(B) = c(D),
\]
where $c(D)$ denotes the crossing number of $D$.  The last equality follows from an euler characteristic calculation.  Equality holds in the first inequality if and only if $\sigma(W)$ and $\sigma(B)$ have opposite signs, and equality holds in the second inequality if and only $W$ and $B$ are definite.  Therefore, $|b(D) - a(D)| = c(D)$ if and only if $W$ and $B$ are definite and of opposite signs.  On the other hand, this equality holds if and only if the connected diagram $D$ is alternating.
The statement of the Proposition now follows.
\end{proof}

\begin{proof}[Proof of Theorem \ref{t: alt definite}]
As in Lemma \ref{l: intersect}, set $X = Y \setminus \overset{\circ}{\nu}(L)$ and put $S_+ \cap X$ and $S_- \cap X$ in minimal position.    Write $\del X = \del_1 X \cup \cdots \cup \del_m X$ corresponding to $L= K_1 \cup \cdots \cup K_m$.  The number of points of intersection in $S_+ \cap S_- \cap \del_i X$ equals the difference in framings $\frac 12 |[K_i]|_{S_+} - \frac 12 |[K_i]|_{S_-}$.  We stress that this difference is non-negative, due to the signs of the surfaces.  The number of arc components of $S_+ \cap S_- \cap X$ equals half the sum of these differences, which is $c : =\frac 12 e(S_+)- \frac 12 e(S_-)$.

An orientation on $X$ induces an orientation on $\del X$, and an orientation on each link component $K_i$ induces orientations on $S_+ \cap \del_i X$ and $S_- \cap \del_i X$.  Every intersection point between $S_+ \cap \del X$ and $S_- \cap \del X$ on $\del X$ has the same sign with respect to these orientations, since $\frac 12 |[K_i]|_{S_+} - \frac 12 |[K_i]|_{S_-} \ge 0$ for all $i$.  An arc component of $S_+ \cap S_- \cap X$ extends to an arc $a \subset S_+ \cap S_-$ such that $a \cap L = \del a$.  Let $A$ denote the union of these $c$ arcs.  It follows from the consistency of the signs of intersection that a neighborhood $\nu(a)$ is modeled on the neighborhood of a crossing in a link diagram, where the checkerboard surfaces meet along an arc that runs between the over and under crossing.  In particular, $\nu(a)$ has a product structure $D^2 \times I$ such that $a$ is contained in $\{0\} \times I$ and the projection to $D^2$ maps $(S_+ \cup S_- -a) \cap \nu(a)$ homeomorphically to $D^2 - \{0\}$.

By Lemma \ref{l: intersect}, $S_+ \cap S_- \cap X$ does not contain any simple closed curves.  Therefore, the 2-complex $S_+ \cup S_-$ is a 2-manifold away from $A$.  From the decomposition of $\nu(S_+ \cup S_-)$ as the union $\nu(A) \cup \nu(S_+ \cup S_- -A)$, we see that $\nu(S_+ \cup S_-)$ can be identified with $\nu(S) \approx S \times I$ for some closed embedded surface $S \subset Y$.  Moreover, the projection $\nu(S) \to S$ maps each arc $a \subset A$ to a distinct point in $S$, and it maps $S_+ \cup S_- - A$ homeomorphically to the complement of these points in $S$.
 
The intersection $S_+ \cap S_- = L \cup A$ has euler characteristic $-c$.  As in the proof of Proposition \ref{p: diagram}, Theorem \ref{t: gl} gives $c = \sigma(S_+)-\sigma(S_-) = b_1(S_+) + b_1(S_-)$.
Thus,
\[
\chi(S_+ \cup S_-) = \chi(S_+) + \chi(S_-) - \chi(S_+ \cap S_-) = (1 - b_1(S_+)) + (1 - b_1(S_-)) + c = 2,
\]
and
\[
\chi(S) = \chi(\nu(S)) = \chi(\nu(S_+ \cup S_-)) = \chi(S_+ \cup S_-) = 2
\]
in turn.
Since $Y$ is orientable and $\chi(S) > 0$, it follows that $S$ contains a sphere component $S_0$.  The neighborhood $\nu(S_0) \subset \nu(S)$ meets $L$ in a sublink $L_0$.  Since $Y \setminus L$ is irreducible, each boundary component of $\nu(S_0) \approx S_0 \times I$ bounds a ball in $Y \setminus L$.  Denoting the balls by $B_1$ and $B_2$, we obtain $Y = B_1 \cup \nu(S_0) \cup B_2 \approx S^3$, $L = L_0$, and $S = S_0$ is a sphere.

We now see that the projection $\nu(S) \to S$ gives a diagram $D$ of $L$.  The $c$ double points of $D$ are the images of the components of $A$, and the chessboard surfaces are isotopic rel boundary to $S_+$ and $S_-$.  Since $S_+$ and $S_-$ are definite and of opposite signs, Proposition \ref{p: diagram} implies that $D$ is alternating, and the characterization is complete.
\end{proof}

\begin{proof}[Proof of Corollary \ref{c: seifert}]
Let $L$ be a special alternating link, $S$ a minimum genus Seifert surface for $L$, $D$ a special alternating diagram of $L$, and $S_D$ the surface obtained by applying Seifert's algorithm to $D$.  Then $S_D$ is one of the spanning surfaces associated with $D$, so it is definite by Proposition \ref{p: diagram}.  Since $e(S) = e(S_D) = 0$ and $S$ has minimum genus, Proposition \ref{p: min complexity} implies that $b_1(S) = b_1(S_D)$ and that $S$ is definite.  It follows that $S_D$ has minimum genus, and by Theorem \ref{t: alt definite}, that $S$ is a spanning surface associated with some (special) alternating diagram of $L$.
\end{proof}


\section{Lattices, graphs, and Tait's conjecture.}
\label{s: tait}

Let $D$ denote a connected alternating diagram of a link $L$.  Color its regions according to the convention that every crossing has type $b$.  Let $B$ and $W$ denote its associated chessboard surfaces.  By the proof of Proposition \ref{p: diagram}, $B$ is negative definite and $W$ is positive definite.  The Gordon-Litherland pairing on either surface admits a natural interpretation as the lattice of integer-valued flows on a graph, as we now recall.

The surface $W$ deform retracts onto a graph $G$ that has a vertex in each white region and an edge through each crossing of $D$.  This is the {\em Tait graph} of $D$.  It has a plane embedding determined up to planar isotopy by $D$, and $c(D) = |E(G)|$.  By the same construction, the surface $B$ deform retracts onto the planar dual $G^*$.  If $D$ is a connected diagram, then $G$ is connected as well.  The diagram $D$ is reduced if and only if both of $G$ and $G^*$ are bridgeless.

Orient the edges of $G$ arbitrarily to endow it with the structure of a 1-dimensional CW-complex.  The chain group $C_1(G)$ inherits the structure of a standard Euclidean lattice by declaring the chosen oriented edge set to form an orthonormal basis.  The {\em flow lattice} $F(G)$ is the sublattice $\ker(\del) \subset C_1(G)$, where $\del : C_1(G) \to C_0(G)$ denotes the boundary operator.  Since $C_2(G)=0$, we can identify the underlying abelian group of $F(G)$ with $H_1(G)$.  The deformation retraction from $W$ to $G$ induces an isomorphism $H_1(W) \approx H_1(G)$. Gordon and Litherland showed that this isomorphism induces an isometry of lattices \cite[Theorem 1]{gl:sig}:

\begin{thm}
\label{t: flows}
Let $D$ denote an alternating diagram, $W$ its white chessboard surface, and $G$ its Tait graph.
Then the deformation retraction from $W$ to $G$ induces an isometry between $(H_1(W), \langle \, , \rangle_W)$ and $F(G)$. \qed
\end{thm}

\noindent
Similarly, the deformation retraction from $B$ to $G^*$ induces an isometry between $(H_1(B), - \langle \, , \rangle_B)$ and $F(G^*)$; we stress the negative sign taken on the intersection pairing on $H_1(B)$.

We obtain the following addendum to Theorem \ref{t: alt definite}.

\begin{cor}
\label{c: reduced}
The surfaces stipulated in Theorem \ref{t: alt definite} do not contain a homology class of self-pairing $\pm 1$ if and only if the alternating diagram guaranteed by Theorem \ref{t: alt definite} is reduced.
\end{cor}

\begin{proof}
The elements of self-pairing $1$ in $F(G)$ are the loops in $G$.  A loop in $G$ is dual to a bridge in $G^*$.  Therefore, $D$ is reduced if and only if $H_1(B)$ and $H_1(W)$ do not contain elements of self-pairing $\pm 1$.  
\end{proof}

An element $v$ in a positive definite lattice $L$ is {\em irreducible} if $v \cdot x < x \cdot x$ for all $x \in L$, $x \ne 0$, and it is {\em simple} if $v \cdot x \le x \cdot x$ for all $x \in L$.  An irreducible element is therefore simple, as is the zero element.  Irreducibility and simplicity are isometry invariants.  The following result is elementary.  The first assertion appears as \cite[Theorem 14.14.4]{godsilroyle:book}, and the second follows as well from its proof.

\begin{prop}
\label{p: irred}
The irreducible elements in $F(G)$ are the oriented cycles in $G$, and the simple elements in $F(G)$ are the oriented Eulerian subgraphs of $G$. \qed
\end{prop}

\noindent
Given Proposition \ref{p: irred}, the proof of the following Lemma is elementary and left to the reader.

\begin{lem}
\label{l: simple}
Suppose that $C_i$ and $C_j$ are oriented cycles in a graph.
The following are equivalent:
\begin{enumerate}
\item
$C_i+C_j$ is simple;
\item
$C_i$ and $C_j$ induce opposite orientations on every edge in $C_i \cap C_j$;
\item
$|E(C_i) \cap E(C_j) | = -C_i \cdot C_j$. \qed
\end{enumerate}
\end{lem}

\begin{thm}
\label{t: torelli}
If $G$ and $G'$ are bridgeless planar graphs with isometric flow lattices, then $|E(G)| = |E(G')|$.
\end{thm}

\begin{proof}
An orientation on $S^2$ induces an orientation on the faces of $G$.  Their oriented boundaries form a collection of oriented cycles $C_1,\dots,C_f \subset G$. They generate $F(G)$ subject to the single relation $C_1+\cdots+C_f = 0$.  Since $G$ is bridgeless, each oriented edge occurs once in the boundary of some $C_i$, and we have
\[
|E(G)|=\sum_{i < j} |E(C_i) \cap E(C_j)|.
\]
It follows as well from Lemma \ref{l: simple} that $C_i + C_j$ is simple for all $i \ne j$.

Suppose that $F(G) \overset{\sim}{\longrightarrow} F(G')$ is an isometry.  The elements $C_1,\dots,C_f$ are irreducible, so their images are oriented cycles $C_1',\dots,C_f' \subset G'$, and $C_i' + C_j'$ is simple for all $i \ne j$.  It follows that no three distinct cycles $C_i', C_j', C_k'$ have an edge in common, since two of them would have to induce the same orientation on it, in violation of Lemma \ref{l: simple}.  Therefore,
\[
\sum_{i < j} |E(C_i') \cap E(C_j')| \le |E(G')|.
\]
On the other hand,
Lemma \ref{l: simple} gives
\[
\sum_{i < j} |E(C_i) \cap E(C_j)| = \sum_{i < j} - C_i \cdot C_j = \sum_{i < j} -C_i' \cdot C_j' = \sum_{i < j} |E(C_i') \cap E(C_j')|.
\]
Combining the indented equations yields $|E(G)| \le |E(G')|$.  By symmetry, the statement of the Theorem follows.
\end{proof}

\begin{proof}[Proof of Theorem \ref{t: tait}]
Let $D$ and $D''$ denote two connected, reduced, alternating diagrams of the same link $L$.  Color them according to the convention that every crossing has type $b$.  Let $W$ denote the white checkerboard surface for $D$ and $B''$ the black checkerboard surface for $D''$.  By Theorem \ref{t: alt definite} and Corollary \ref{c: reduced}, there exists a reduced, alternating diagram $D'$ whose white checkerboard surface $W'$ is isotopic to $W$ and whose black checkerboard surface $B'$ is isotopic to $B''$.  Let $G$ denote the Tait graph of $D$ and $G'$ the Tait graph of $D'$.  Since $W \simeq W'$, it follows from Theorem \ref{t: flows} that $F(G) \approx F(G')$.  By Theorem \ref{t: torelli}, it follows that $c(D) = |E(G)|=|E(G')| = c(D')$.  Similarly, $c(D') = c(D'')$, and the first part of the Theorem follows.

For the second part, orient the diagrams.  Let $p(\cdot)$ and $n(\cdot)$ denote the number of positive and negative crossings in a diagram, respectively.
By Theorem \ref{t: gl}, \cite[Lemma 7]{gl:sig}, and the fact that every crossing has type $b$, it follows that $\sigma(L) = \sigma(W)-p(D) = \sigma(W')-p(D')$.  Since $W \simeq W'$, it follows that $p(D) = p(D')$.  Since $c(D) = c(D')$ by the first part of the Theorem, it follows that $n(D) =n(D')$, as well.  Therefore, $D$ and $D'$ have the same writhe $p(D)-n(D) = p(D')-n(D')$.  Similarly, $D'$ and $D''$ have the same writhe, and the second part of the Theorem follows.
\end{proof}


\section{Appendix: Algorithmic Detection of Alternating Links. \\ {\small By Andr\'as Juh\'asz and Marc Lackenby}}

This appendix contains the proof of Theorem \ref{t: algorithm}.

We are given a diagram $D$ for a prime knot $K$ with $c$ crossings.  If the knot $K$ is alternating, then it has a reduced alternating diagram having $c' \leq c$ crossings, by a theorem of Kauffman, Murasugi, and Thistlethwaite \cite{kauff:altknots, murasugi:altknots, thistle:spanning}. The chessboard surfaces $S_1$ and $S_2$ for this diagram have the following properties: \begin{inparaenum} \item $\chi(S_1) + \chi(S_2) = 2 - c' \geq 2 - c$; \item the Gordon-Litherland pairings of $S_1$ and $S_2$ are positive definite and negative definite, respectively; \item they are incompressible, boundary-incompressible, and $\pi_1$-injective \cite{aumann}.  \end{inparaenum}  Conversely, by Theorem \ref{t: alt definite}, the existence of such spanning surfaces $S_1$ and $S_2$ satisfying (2) imply that the knot is alternating. We need to show how to find these surfaces. We start by using $D$ to construct a triangulation $T$ of $X$, the exterior of $K$, with the property that a meridian is a subset $\Gamma$ of the 1-skeleton. The number of tetrahedra in $T$ can be bounded above by a linear function of $c$. Then we need the following result.

\begin{lem}
\label{l: normal}
The surfaces $S_1$ and $S_2$ can be realised as normal surfaces with respect to $T$. Each is a sum of at most $c$ fundamental normal surfaces.  The number of normal triangles and squares in $S_1$ and $S_2$ is at most an exponential function of $c$.
\end{lem}

\begin{proof}
This is a fairly well-known application of normal surface theory. We give the manifold $X$ the boundary pattern $\Gamma$. Then $(X, \Gamma)$ is simple in the sense of \cite[Definition 6.3.16]{matveev}. This is because of Menasco's theorem that the exterior of a prime alternating knot contains no essential torus and the only essential annuli arise as the obvious annuli for a $(2,n)$ torus knot, but these necessarily intersect $\Gamma$ \cite[Corollary 2]{menasco:alternating}. In [Mat03, Definition 6.3.8], Matveev defines the $p$-complexity of the normal surface $S_i$ to be $-\chi(S_i) + |S_i \cap \Gamma|$, and this is at most a linear function of $c$, by (1) above.  By \cite[Theorem 6.3.17]{matveev}, one can construct a finite list of normal surfaces with the property that any 2-sided properly embedded incompressible, boundary-incompressible, connected surface with at most this $p$-complexity is strongly equivalent to one in this list. Here, strongly equivalent just means that there is a homeomorphism of the pair $(X, \Gamma)$ taking it to one of these normal surfaces.  Now, the surfaces $S_i$ need not be 2-sided, but the proof of \cite[Theorem 6.3.17]{matveev} gives that $S_i$ is a sum $\sum \lambda_j F_j$ of fundamental normal surfaces. In fact, the only $F_j$ that appear in the sum have positive $p$-complexity. (No normal tori appear in the sum, using \cite[Proposition 6.3.21]{matveev}.) Hence, the number of summands for $S_i$ is at most the $p$-complexity, which is at most $c$ by (1). By a result of Hass and Lagarias, the number of triangles and squares in a fundamental normal surface is at most an exponential function of the number of tetrahedra \cite[Lemma 2.3]{hl2001}. Hence, this gives the final part of the Lemma.
\end{proof}

Assuming Lemma 6.1, the algorithm simply constructs all such normal surfaces. Since the number of triangles and squares in each surface $S_i$ is at most an exponential function of $c$, and the number of triangle and square types is at most a linear function of $c$, the number of possible normal surfaces we must consider is at most an exponential function of $c^2$. For each surface $S_i$, one has an explicit decomposition of the surface into triangles and squares. From this, one can find a spanning set for $H_1(S_i)$ in the 1-skeleton of $S_i$. One can then reduce this to a basis for $H_1(S_i)$ using linear algebra. The size of this basis is at most a linear function of $c$, because of (1). So we can compute the Gordon-Litherland pairing and then determine whether it is positive or negative definite.

\bibliographystyle{myalpha}
\bibliography{/Users/greenegh/Dropbox/Papers/References}

\end{document}